\documentclass[letterpaper, 11pt]{article}
\usepackage{amsmath, amsthm, amssymb}
\usepackage{thm-restate}
\usepackage{graphicx}
\usepackage{xcolor}
\usepackage{algorithm,algorithmic}
\usepackage[colorlinks,linkcolor=blue,citecolor=blue,urlcolor=magenta,linktocpage,plainpages=false]{hyperref}
\usepackage{caption}
\usepackage{subcaption}
\usepackage{setspace}
\usepackage[left=1in, right=1in, top=1in, bottom=1in]{geometry}
\usepackage{authblk}
\usepackage{changepage}
\usepackage{enumitem}

\newtheorem{proposition}{Proposition}
\newtheorem{notation}{Notation}

\newtheorem{remark}{Remark}

\newcommand{\Real}{\mathbb{R}}
\newcommand{\Ze}{\mathbb{Z}}
\newcommand{\conv}{\mathrm{conv}}

\def\nni{{non-negative instances}}
\def\Nni{{Non-negative instances}}
\def\msi{{mixed-signs instances}}
\def\Msi{{Mixed-signs instances}}

\title{%\vspace{-1cm}
Solving sparse separable bilinear programs using lifted bilinear cover inequalities
}
\author{
Xiaoyi Gu\thanks{xiaoyigu@gatech.edu, H. Milton  Stewart School of Industrial \& Systems Engineering, Georgia Institute of Technology, Atlanta, GA 30332.},
Santanu S. Dey\thanks{santanu.dey@isye.gatech.edu, H. Milton  Stewart School of Industrial \& Systems Engineering, Georgia Institute of Technology, Atlanta, GA 30332.}, 
Jean-Philippe P. Richard\thanks{jrichar@umn.edu, Department of Industrial and Systems Engineering, University of Minnesota.} 
} 

\date{}

\begin{document}

\maketitle
\begin{abstract} 
Recently, in \cite{gu2022lifting}, we proposed a class of inequalities called \textit{lifted bilinear cover inequalities}, which are second-order cone representable convex inequalities, and are valid for a set described by a separable bilinear constraint together with bounds on variables. In this paper, we study the computational potential of these inequalities for separable bilinear optimization problems. 
We first prove that the semi-definite programming relaxation provides no benefit over the McCormick relaxation for such problems. 
We then design a simple randomized separation heuristic for lifted bilinear cover inequalities. 
In our computational experiments, we separate many rounds of these inequalities starting from McCormick's relaxation of instances where each constraint is a separable bilinear constraint set. 
We demonstrate that there is a significant improvement in the performance of a state-of-the-art global solver in terms of gap closed, when these inequalities are added at the root node compared to when they are not.
\end{abstract}

\section{Introduction.}

{Lifting} is a technique used to generate cutting planes for a set from a {seed inequality} valid for a restriction of this set.
Lifting was first studied in the context of mixed integer linear programming (MILP); see~\cite{richard2010lifting} for a review.  
Cover inequalities (\cite{wolsey1975faces,padberg1975note,balas1975facets,balas1978facets,hammer1975facet}) are valid for 0-1 knapsack sets with coefficients satisfying the minimal cover property. 
The inequalities obtained after lifting cover inequalities are called \textit{lifted cover inequalities}; see ~\cite{hammer1975facet,zemel1978lifting,wolsey1976facets,wolsey1977valid,gu1999lifted,gu2000sequence}.
They form a family of strong valid inequalities for general 0-1 knapsack sets that is very important for state-of-the-art MILP solvers; see ~\cite{bixby2007progress}. 

Inspired by the success of these inequalities we introduced in~\cite{gu2022lifting} a class of inequalities we call \textit{lifted bilinear cover inequalities} for separable bilinear constraints.  
These second-order cone representable (SOCR) convex inequalities are derived using lifting and are valid for a set described by a separable bilinear constraint together with bounds on variables. 
In Section~\ref{sec:background} below, we give a detailed description of lifted bilinear cover inequalities and the separable bilinear sets they are valid for. 

In this paper, we study the computational potential of these inequalities. 
We are inspired by the classical paper of~\cite{crowder1983solving}. 
This paper is one of the first to highlight the computational importance of lifted cover inequalities in MILP. 
Specifically, this paper considered sparse instances -- since most practical instances are sparse. 
Here, sparsity means that the support of each constraint is significantly smaller than the number of variables in the entire problem; see~\cite{dey2015approximating,dey2018analysis} for discussions and results regarding sparsity of problems and cutting-planes. 
Similarly, we generate sparse separable bilinear instances to test lifted bilinear cover inequalities. 
We numerically show that a significant improvement in the performance (in terms of gap closed) of a global solver can be observed when these inequalities are added at the root node compared to when they are not.

\subsection{Lifted bilinear cover inequalities for separable bilinear programs.}\label{sec:background}
For a positive integer $n$, we use the notation $[n]$ to describe the set $\{1, \dots, n\}$.

In \cite{gu2022lifting}, we derive inequalities that can be applied to improve convex relaxations of the feasible region of separable bilinear programs, which we call inequalities \textit{lifted bilinear cover inequalities}.

\begin{restatable}{definition}{DefBilinearProblem}
A separable bilinear program is an optimization problem of the form
\begin{equation}
\label{eq:bilinearproblem}
\begin{aligned}
    \textup{min}\quad & \sum_{i \in [n]} \left( c_i^xx_i + c_i^y y_i\right)\\
    \textup{s.t.}\quad & \sum_{i \in [n]} a_i^j x_iy_i \geq d_j,\quad && \forall j \in [m],\\
    & x_i, y_i \in [0, 1], \quad && \forall i \in [n].
\end{aligned}
\end{equation}
\end{restatable}

Each bilinear constraint in \eqref{eq:bilinearproblem} is in the form of a \textit{separable bilinear set}, defined as follows.

\begin{restatable}{definition}{DefSeparable}
\label{def:separable}
A set $S$  is called a \textit{separable bilinear set} if it is of the form
\begin{eqnarray*}
S:= \left\{ \ (x , y) \in [0, 1]^n \times [0, 1]^n \ \Bigm|\  \sum_{i \in [n]} a_i x_iy_i \geq d \ \right\},
\end{eqnarray*}
where $d\in \Real$ and $a_i \in \Real$ for all $i \in [n]$.
\end{restatable}

For each $i \in [n]$, variables $x_i$ and $y_i$ in $S$  appear in only one term in the left-hand-side. 
The convex hull of the set $S$ has been studied in~\cite{dey2019new,santana2020convex} and the convex hull of the relaxation of $S$ obtained by dropping the upper bounds on the variables is presented in~\cite{tawarmalani2010strong} for the case where all coefficients $a_i$ are non-negative. 
These convex hull results (especially the exact ones presented in~\cite{dey2019new} and \cite{santana2020convex}) have limited computational usefulness, since the description of the convex hull is exponential in the number of variables. 
This motivates the derivation of families of cutting planes for $S$. 

In \cite{gu2022lifting}
we proposed using \textit{minimal covers} to generate seed inequalities we call \textit{bilinear cover inequalities} and then perform lifting to obtain valid inequalities for separable bilinear sets. 
We next briefly introduce the main results of \cite{gu2022lifting}. 
We start by introducing the notions of \textit{minimal cover} and \textit{minimal cover yielding partition}.

\begin{restatable}{definition}{DefMinimal}
\label{def:minimal}
A set $\{a_i\in \Real \ |\ i\in [k]\}$, with $k \in \Ze_{+}$ being a positive integer, is said to form a \textit{minimal cover} of $d\in \Real$ if
\begin{enumerate}[label={(\roman*)},align=left]
    \item $a_i > 0$ for all $i \in [k]$, $d >0$,
    \item $\sum_{i \in [k]} a_i > d$,
    \item $\sum_{i \in K} a_i \leq d$, $\forall K \subsetneq [k]$.
\end{enumerate}
For a separable  bilinear set $S$,  
we say that a partition $\Lambda = \{I, J_0, J_1\}$ of $[n]$, where $I\neq \emptyset$, is a \textit{minimal cover yielding partition} if: $\{a_i\ |\  i \in I\}$ forms a minimal cover of $d^\Lambda:= d - \sum_{i\in J_1} a_i$. 
For a minimal cover yielding partition, we let $J_0^+ := \{ i\in J_0\ |\ a_i>0 \}$ and $J_0^- := \{ i\in J_0\ |\ a_i < 0 \}$. 
We define $J_1^+$ and $J_1^-$ similarly.
\end{restatable}

\begin{remark}
When $k \geq 2$, conditions (ii) and (iii) in the definition of minimal cover imply condition (i). For example,  if $a_i \leq 0$ for some $i\in [k]$, then (ii) implies $\sum_{j \in [k] \setminus \{i\}} a_j > d$, contradicting  (iii).  Now (iii) together with $a_i >0$ for $i \in [k]$ implies that $d >0$. 
\end{remark}

\begin{notation}\label{not:1}
Assuming that $\{a_i\ | i\in [k] \}$ forms a minimal cover of ${d^\Lambda}$, we use 
\begin{enumerate}
\item $\Delta:= \sum_{i \in [k]} a_i - {d^\Lambda}$, 
\item $d_i:= {d^\Lambda} - \sum_{j\in [k]\backslash \{i\}} a_j = a_i - \Delta$, 
\item $I^> := \{i \in [k] \ |\  a_i > \Delta\}$, 
\item when $I^> \neq \emptyset$, let $\i0$ denote any index in $I^>$ such that $a_{\i0} = \min \{ a_i \ |\  i \in I^> \}$.
{We say that $i_0$ does not exist when $I^>=\emptyset$.}
\end{enumerate}
\end{notation}

The process to generate lifted bilinear cover inequalities for $S$ from the minimal cover yielding partition is to 
(i) first fix $x_i = y_i=0$ for $i\in J_0$ and $x_i=y_i=1$ for $i\in J_1$, 
(ii) generate valid seed inequality for the restricted region, and 
(iii) finally lift the seed inequality to obtain a valid cut for $S$.

The following result from \cite{gu2022lifting} establishes the existence of minimal cover yielding partitions in many instances of $S$.

\begin{restatable}{theorem}{ThmMinimal}
\label{thm:minimal}
For a nonempty separable bilinear set $S$, either there exists at least one minimal cover yielding partition or $\conv(S)$ is polyhedral.
\end{restatable}

The seed inequality used for lifting is presented in Theorem~\ref{thm:valid}, which also discusses its strength.

\begin{restatable}{theorem}{ThmValid}
\label{thm:valid}
For a separable  bilinear set $S$ as in Definition~\ref{def:separable}
where $\{a_i\ |\ i\in [n]\}$ forms a minimal cover of $d$, the following \textit{bilinear cover inequality} is valid:
\begin{eqnarray} \label{eq:bilincoverineq}
\sum_{i \in [n]} \frac{\sqrt{a_i}}{ \sqrt{a_i}-\sqrt{d_i}}\left( \sqrt{x_iy_i} - 1\right) \geq -1.
\end{eqnarray}
Further, the set $R:= \{ (x,y) \in \Real^{2n}_{+} \ |\  (\ref{eq:bilincoverineq})\}$
is such that
$ (4\cdot R) \cap [0, \ 1]^{2n} \subseteq \textup{conv}(S) \subseteq R \cap [0, \ 1]^{2n}.$
\end{restatable}

With the seed bilinear cover inequality, the following lifted bilinear cover inequality was proposed, which is valid for the whole separable bilinear set.

\begin{restatable}{theorem}{ThmLifted}
\label{thm:lifted}
Consider a separable  bilinear set $S$ as in Definition~\ref{def:separable}. 
Let $\Lambda = \{I, J_0, J_1\}$ be a minimal cover yielding partition, let $\Delta, a_{\i0}, d_i, l_{+}, l_{-}$ be defined as in {Notation~\ref{not:1}}, and let $J_0^+$, $J_0^-$, $J_1^+$, $J_1^-$ be as in Definition~\ref{def:minimal}. 
Then inequality 
\begin{eqnarray}
\label{eq:liftedbilinearcoverinequality}
\sum_{i \in I} \frac{\sqrt{a_i}}{ \sqrt{a_i}-\sqrt{d_i}}\left( \sqrt{x_iy_i} - 1\right) + \sum_{i\notin I} \gamma_i (x_i, y_i) \geq -1,
\end{eqnarray}
is valid for $S$ where 
$\gamma_i: \Real^2 \rightarrow \Real$ for $i \in [n]\setminus I$ are the concave functions:
\begin{enumerate}[label={(\roman*)},align=left]
    \item $\gamma_i(x,y) = l_+a_i\min\{x,y\}$ for $i\in J_0^+$;
    \item $\gamma_i(x,y) = -l_+a_i\min\{2-x-y, 1\}$ for $i\in J_1^-$;
    \item $\gamma_i(x,y) = \min\{l_-a_i(x+y-1), l_+a_i(x+y-1)+l_+\Delta-1, 0\}$ for $i\in J_0^-$;
    \item 
    $\gamma_i(x,y) = \min\{\tilde{g}_i(x,y),\tilde{h}_i(x,y), g_i(x,y), h_i(x,y)\},$
    for $i\in J_1^+$ with $a_i\geq a_{\i0}$ when $I^>\neq\emptyset$, 
    and
    $\gamma_i(x,y) = \min\{\tilde{g}_i(x,y),\tilde{h}_i(x,y)\}$ in all other cases where 
    $i\in J_1^+$, 
    with
    \begin{align*}
        \tilde{g}_i(x,y)&=l_+a_i(\min\{x, y\}-1)+l_+\Delta-1\\
        \tilde{h}_i(x,y)&=l_-a_i(\min\{x, y\}-1)\\
        g_i(x,y)&= \sqrt{a_i - \Delta}\sqrt{a_i}l_+\sqrt{xy} -l_+(a_i-\Delta) -1 \\
        h_i(x,y)&= \frac{\sqrt{a_i}}{\sqrt{a_i} - \sqrt{d_i}}(\sqrt{xy} - 1)
    \end{align*}
\end{enumerate}
with $l_- = \frac{1}{\Delta}$ and
$l_+=\frac{\sqrt{a_{i_0}}+\sqrt{d_{i_0}}}{\Delta \sqrt{d_{i_0}}}$
if 
{$i_0$} exists and $l_+=\frac{1}{\Delta}$ otherwise.
\end{restatable}

The lifted bilinear cover inequality \eqref{eq:liftedbilinearcoverinequality} is convex and second-order cone representable (SOCR), making it easy to incorporate in relaxations, given the enormous progress in SOC solvers. 

\subsection{Main contributions.}
While the lifted bilinear cover inequality was derived in \cite{gu2022lifting}, its computational usefulness has not been evaluated.  In this paper, we utilize lifted bilinear cover inequalities in a computational study and illustrate their benefit. 

We expand on our main contributions next:
\begin{itemize}
\item 
Instead of using lifted bilinear cover inequalities, we could consider using the semidefinite programming (SDP) relaxation for strengthening the McCormick relaxation. In general, it is known that combining the McCormick relaxation with the SDP relaxation produces good bounds; see~\cite{burer2015gentle} for instance.  We show that, surprisingly, for the class of separable  instances, the SDP relaxation gives a bound equal to that of the McCormick relaxation. We shall see in the later sections that the lifted bilinear cover inequalities are able to close significant root gap over the McCormick relaxation -- thus showing that lifted bilinear cover inequalities are important in solving separable bilinear instances where SDP relaxations are of limited use.
\item 
We design a simple randomized separation heuristic for lifted bilinear cover inequalities. We use this separation heuristic to add many rounds of cuts starting from a natural relaxation of the problem that uses McCormick inequalities.  We then solve the instances using a commercial global solver with and without these inequalities added to the root node and compare the results obtained.
\item 
We discover that the inequalities separated using the simple heuristic provide a major performance boost to a commercial global solver in terms of overall gap closed on sparse test instances.
\end{itemize}

The rest of the paper is organized as follows. 
In Section~\ref{sec:SDP}, we compare the SDP and McCormick relaxations for separable  programs. 
In Section~\ref{sec:heur}, we present our separation heuristic. 
In Section~\ref{sec:setup}, we describe the procedure we use to generate synthetic test instances and discuss various other parameters regarding the use of the separation heuristic on these instances. 
In Section~\ref{sec:eval}, we present detailed results of our experiments. 
We give concluding remarks in Section~\ref{sec:conc}.

\section{Comparing SDP and McCormick relaxations of separable  programs.}\label{sec:SDP}

We show that, for separable bilinear programs, the relaxation obtained by adding the traditional SDP constraint to the McCormick relaxation gives the same bound as that given by just the McCormick relaxation. 
In fact, we prove this result in the following slightly more general seting. 
Consider the following quadratically constrained quadratic program (QCQP):
\begin{subequations}
\label{prob:bilin}
\begin{alignat}{5} 
&\textup{min}&\quad & \sum_{i \in [n]} a^0_{i} x_{i}y_i  + \sum_{i\in [n]} c^{x,0}_ix_i + \sum_{i \in [n]}c^{y,0}_iy_i\label{eq:obj}\\
&\textup{s.t.}&& \sum_{i \in [n]} a^j_{i} x_{i}y_i  + \sum_{i\in [n]} c^{x,j}_ix_i + \sum_{i \in [n]}c^{y,j}_iy_i \geq b_j, \quad \forall j \in [m], \label{eq:con}\\
&&& x, y \in [0, 1]^n.\label{eq:bnds}
\end{alignat}
\end{subequations}
We consider next two relaxations of \eqref{prob:bilin}.
{In both of these relaxations, we introduce variables $w_{i,j}$ to represent the products $x_i x_j$, variables $w_{i,n+j}$ and $w_{n+i,j}$ to represent the products $x_i y_j$ and $y_i x_j$, respectively, and variables $w_{n+i,n+j}$ to represent the products $y_i y_j$. 
Clearly, only variables $w_{i,n+i}$ are needed to relax \eqref{eq:con}. The first relaxation, which we call \textit{McCormick relaxation}, is obtained by using McCormick inequalities to approximate the relationships between variables $w$, $x$, and $y$: 
}
\begin{subequations}
\label{prob:mc}
\begin{alignat}{5} 
z_{Mc}:=\ &\textup{min}&\quad& \sum_{i \in [n]}a^0_{i} w_{i, n + i}  + \sum_{i\in [n]} c^{x,0}_ix_i + \sum_{i \in [n]}c^{y,0}_iy_i
& \quad & \\
&\textup{s.t.}&&\sum_{i \in [n]} a^j_{i} w_{i, n + i}  + \sum_{i\in [n]} c^{x,j}_ix_i + \sum_{i \in [n]}c^{y,j}_iy_i \geq b_j, && \forall j \in [m], \label{eq:eq1}\\
&&& x_i = u_i, && \forall i \in [n], \label{eq:eq2}\\
&&& y_i = u_{n + i}, &&  \forall i \in [n], \label{eq:eq3}\\
&&& u \in [0, 1]^{2n}, \label{eq:eq4}\\
&&& \textup{max}\{0, u_i + u_k  - 1\} \leq w_{i,k} \leq \textup{min}\{u_i , u_k\}, && \forall(i,k) \in [2n]\times[2n]. \label{eq:eq5}
\end{alignat}
\end{subequations}
{The second relaxation, which we call \textit{McCormick+SDP relaxation}, is obtained from the McCormick relaxation by including the traditional SDP relaxation of the property that $W$ is a rank-1 matrix:}
\begin{subequations}
\label{prob:sdp}
\begin{alignat}{5} 
z_{MS}:=\ &\textup{min}&\quad& \sum_{i \in [n]}a^0_{i} w_{i, n + i}  + \sum_{i\in [n]} c^{x,0}_ix_i + \sum_{i \in [n]}c^{y,0}_iy_i
 \\
&\textup{s.t.}&& (\ref{eq:eq1}), (\ref{eq:eq2}), (\ref{eq:eq3}), (\ref{eq:eq4}), (\ref{eq:eq5}) \\
&&& W:= \left[ \begin{array}{cc} 1 &\ u^{\top} \\ u &\ w \end{array} \right] \succeq 0,
\end{alignat}
\end{subequations}
where $W\succeq 0$ denotes that $W$ is positive semi-definite. 

\begin{proposition}
For the optimization problem \eqref{prob:bilin}, $z_{Mc} = z_{MS}$, {or both \eqref{prob:mc} and \eqref{prob:sdp} are infeasible}.
\end{proposition}
\begin{proof}
{Note that both \eqref{prob:mc} and \eqref{prob:sdp} are bounded, as all variables are bounded by $[0,1]$. Therefore, if either problem is feasible, then the feasible region is compact and thus a corresponding optimal solution must exist (objective function is linear). Since the feasible region of \eqref{prob:sdp} is contained in that of \eqref{prob:mc}, we only need to prove that if \eqref{prob:mc} is feasible, then \eqref{prob:sdp} is also feasible and that $z_{Mc} \leq z_{MS}$.} 

Let $(x^*, y^*, u^*, w^*)$ be an optimal solution corresponding to \eqref{prob:mc}.  

Then, construct the matrix $W^* \in \Real^{(2n +1) \times (2n+1)}$ as follows:
\begin{eqnarray*}
W^*_{ik} = \left\{\begin{array}{ll} 1 , &\quad \textup{if } i = k = 1,\\
u^*_{i -1} , &\quad \textup{if } k = 1, i \geq 2,\\
u^*_{k  -1} , &\quad \textup{if } i = 1, k \geq 2,\\
u^*_{i -1} , &\quad  \textup{if } i = k > 1,\\
u^*_{i-1}u^{*}_{k -1}, &\quad \textup{if } i \neq k,  i > 1, k > 1, |i - k| \neq n, \\
w^*_{\textup{min}\{i -1, k - 1\}, \textup{max}\{i -1, k - 1\}} , &\quad \textup{if } i \neq k,  i > 1, k > 1, |i - k| = n.
\end{array}\right.
\end{eqnarray*}

It is easy to verify that $(x^*, y^*, u^*, W^*)$ satisfy all the linear constraints corresponding to $z_{MS}$. 
It remains to show that $W^* \succeq 0$. 
This would show that there exists a feasible solution to \eqref{prob:sdp} with objective function value equal to that of $z_{Mc}$, completing the proof. 

Construct the matrix $G\in \Real^{(2n +1) \times (2n+1)}$ as follows:
\begin{eqnarray*}
G_{ik} = \left\{\begin{array}{ll} 1, &\quad \textup{if } i = k = 1,\\
u^*_{i -1} , &\quad \textup{if } k = 1,\\
u^*_{k  -1} , &\quad \textup{if }  i = 1,\\
u^*_{i-1}u^{*}_{k -1} , &\quad \textup{if } j \neq k,  i > 1, k > 1.
\end{array}\right.
\end{eqnarray*}
Clearly, {$G\succeq 0$} 
since $G = \left [\begin{array}{c} 1 \\ u^* \end{array}\right][1 \quad (u^*)^{\top}]$. 
So it is sufficient to show that {$W^* - G\succeq 0$.}  We compute that
\begin{eqnarray*}
(W^* - G)_{ik} = \left\{\begin{array}{ll} 0, &\quad \textup{if } i = 1 \textup{ or } k = 1,\\
u^*_{i -1}  - (u^*_{i -1})^2 , &\quad  \textup{if } i = k > 1,\\
0 , &\quad \textup{if } i \neq k,  i > 1, k > 1, |i - k| \neq n, \\
w^*_{\textup{min}\{i -1, k - 1\}, \textup{max}\{i -1, k - 1\}} - u^*_{i-1}u^{*}_{k -1}, &\quad \textup{if } i \neq k,  i > 1, k > 1, |i - k| = n. 
\end{array}\right.
\end{eqnarray*}
In words, the first row and column are all zero.  For each of the remaining rows, there is exactly one non-zero off-diagonal term. 
Therefore (without loss of generality) by the Gershgorin circle theorem (see~\cite{horn2012matrix}) it is sufficient to prove that
$$u^*_{i -1}  - (u^*_{i -1})^2  \geq |w^*_{i -1, n + i - 1} - u^*_{i-1}u^{*}_{i + n  -1}|, \quad \forall i \in \{2, \dots, n + 1\},$$
that is, it is sufficient to prove that
$$x^*_{i}  - (x^*_{i})^2  \geq |w^*_{i,n + i} - x^*_{i}y^{*}_{i} |, \quad \forall i \in [n]$$
where $\textup{min} \{x^*_i, y^*_i \} \geq  w^*_{i,n+i} \geq \textup{max} \{x^*_i + y^*_i  - 1, 0\}$ for $i \in [n]$. 
This can be done by considering the four possible extreme values of $w^*$:

\begin{itemize}
\item {\bf Case 1:}
$w^*_{i,n+ i} = 0$. 
In this case, $x^*_i + y^*_i  \leq 1$. 
Then 
\begin{eqnarray*}
{
1 - x^*_i \geq  y^{*}_{i} \quad \Rightarrow \quad
x^*_i(1 - x^*_i) \geq  x^*_{i}y^{*}_{i}  \quad \Rightarrow \quad
x^*_{i}  - (x^*_{i})^2   \geq  x^*_{i}y^{*}_{i} = |w^*_{i,n + i} - x^*_{i}y^{*}_{i} |.
}
\end{eqnarray*}
\item {\bf Case 2:}
$w^*_{i,n+i} = x^*_i + y^*_i - 1 \geq 0$. 
In this case, $x^*_i + y^*_i \geq 1$. 
First, observe that $w^*_{i,n+i} - x^*_{i}y^{*}_{i}  = x^*_i + y^*_i - 1 - x^*_{i}y^{*}_{i} = (- 1 + y^*_i)(1  - x^*_i) \leq 0$. 
Thus, $|w^*_{i,n+i} - x^*_{i}y^{*}_{i}| = -(- 1 + y^*_i)(1  - x^*_i)$. 
Then,
\begin{eqnarray*}
x^*_i \geq 1 - y^{*}_{i} &\quad \Rightarrow \quad & x^*_i(1 - x^*_i) \geq  -(- 1 + y^*_i)(1  - x^*_i)\\ 
&\quad \Rightarrow \quad & x^*_{i}  - (x^*_{i})^2   \geq -(-1 + y^*_i)(1  - x^*_i) = |w^*_{i,n + i} - x^*_{i}y^{*}_{i} |.
\end{eqnarray*}
\item {\bf Case 3:}
$w^*_{i,n +i} = x^*_{i}$. 
In this case, $x_i^* \leq y_i^*$. 
Then,
\begin{eqnarray*}
{
x^*_{i} \leq  y^*_i \quad \Rightarrow \quad
(x^*_{i})^2 \leq  x^*_iy^*_i   \quad \Rightarrow \quad
x^*_{i}  - (x^*_{i})^2   \geq  x^*_i - x^*_iy^*_i  = |w^*_{i,n + i} - x^*_{i}y^{*}_{i} |.
}
\end{eqnarray*}
\item {\bf Case 4:}
$w^*_{i,n+i} = y^*_{i}$. 
In this case, $x_i^* \geq y_i^*$. 
Then,
\begin{eqnarray*}
{
x^*_{i} \geq  y^*_{i} \quad \Rightarrow \quad
x^*_{i}(1  - x^*_{i}) \geq  y^*_{i}(1  - x^*_{i})    \quad \Rightarrow \quad
x^*_{i}  - (x^*_{i})^2   \geq  y^*_i - x^*_iy^*_i = |w^*_{i,n + i} - x^*_{i}y^{*}_{i} |.
}
\end{eqnarray*}
\end{itemize}
\end{proof}

\section{Improved relaxation through a heuristic separation algorithm.}
\label{sec:heur}

{In this section, we describe a procedure that, starting from the McCormick relaxation of the problem, produces an improved relaxation by adding a selection of lifted bilinear cover inequalities.  This procedure, whose details are presented in Algorithm~\ref{alg:scheme}, solves a convex relaxation of the problem at each step and separates the obtained optimal relaxation solution using a heuristic separation procedure.  The steps are repeated until either a preset time limit is reached, a preset iteration limit is reached, or the improvement in relaxation bounds between iterations becomes too small. }

\floatname{algorithm}{Algorithm}
\begin{algorithm}
	\caption{Improving Root Relaxation with Lifted Bilinear Cover Cuts}
	\label{alg:scheme}
	\begin{algorithmic}[1]
		\REQUIRE bilinear programming problem; non-convex nonlinear solver \texttt{globalsolver}; convex second-order cone programming (SOCP) solver \texttt{convexsolver};\\
		parameters: \\ minimum improvement threshold $\epsilon_z$, iteration number limit $T$, heuristic time limit $T_{heu}$,
		 parameters for heuristic and solver
        \STATE generate McCormick relaxation \eqref{prob:mc} 
\STATE solve \eqref{prob:mc} using  
\texttt{convexsolver};\\
obtain optimal solution $(x^0, y^0)$ and lower bound $z^0$; set $t=1$
        \WHILE{(heuristic time limit $T_{heu}$ not reached)}
        \STATE generate bilinear cover cuts using Algorithm~\ref{alg:heu} and add them to the problem 
        \STATE \textbf{if} (violated bilinear cover cuts were not generated) \textbf{break}
        \STATE solve the improved relaxation with \texttt{convexsolver};\\ obtain optimal solution $(x^t, y^t)$ and lower bound $z^t$
        \STATE \textbf{if} $|z^t-z^{t-1}|/|z^{t-1}|<\epsilon_z$ (improvement is too small), \textbf{break}
        \STATE $t\leftarrow t+1$
        \STATE \textbf{if} $t>T$ (iteration limit is reached) \textbf{break}
        \ENDWHILE
        \STATE solve the problem using 
        \texttt{globalsolver}
	\end{algorithmic}
\end{algorithm}

{The workhorse of the procedure described above is a heuristic separation algorithm that, given a relaxation solution, seeks to generate a violated lifted bilinear cover cut from each of the bilinear constraints of the problem. }
The heuristic, which we explain next, is formally presented in Algorithm~\ref{alg:heu}.

\floatname{algorithm}{Algorithm}
\begin{algorithm}
	\caption{Heuristic Separation Algorithm for Lifted Bilinear Cover Cuts}
	\label{alg:heu}
	\begin{algorithmic}[1]
	\REQUIRE bilinear programming problem with relaxation solution 
	$(x^t, y^t)$;\\
	parameters:\\
	approximation threshold $\epsilon$, search iteration limit $S_j$ for $j \in [m]$
	\FOR {$j=1:m$}
	\STATE \textbf{if} heuristic time limit $T_{heu}$ reached, \textbf{break}
	\STATE \textbf{if} $\sum_{i=1}^n a_i^j x_iy_i-d_j\geq 0$ (bilinear constraint $j$ not violated), \textbf{continue}
	\FOR {$i=1:n$}
	\STATE \textbf{if} $a_i^j=0$, set label $l_i^j$ as inactive ($(x_i, y_i)$ is not in the bilinear cover cut for row $j$)
	\STATE \textbf{else if} $x_iy_i<\epsilon$, set $l_i^j$ as $J_0$
	\STATE \textbf{else if} $x_iy_i>1-\epsilon$, set $l_i^j$ as $J_1$
	\STATE \textbf{else if} $a_i^j>0$, set $l_i^j$ as $I$
	\STATE \textbf{else} (\textit{i.e.}, $a_i^j<0$), set $l_i^j$ as $J_1$ or $J_0$ with probability $x_iy_i$ and $1-x_iy_i$ respectively
	\ENDFOR
	\IF {partition $\Lambda$ with labels $l$ is not a minimal cover yielding partition for row $j$}
	\FOR {$s=1:S_j$} 
	\STATE (check reason for failure)
	{\STATE \textbf{if} $d_j^\Lambda = d_j - \sum_{i\in J_1} a_i^j \leq 0$ (right-hand-side is not positive):\\
 	select randomly $i^* \in \{i\ | \ a_i^j>0, l_i^j=J_1\}\cup \{i \ | \  a_i^j<0, l_i^j = J_0\}$ and flip $l_{i^*}^j$ from $J_1$ to $I$ (in the first case) or from $J_0$ to $J_1$ (in the second case); \textbf{break} if not found
    \STATE \textbf{else if} $\Delta_j = \sum_{I} a_i^j-d_j^\Lambda \leq 0$ ($\Delta_j$ is not positive, \textit{i.e.}, not a cover):\\
	select randomly $i^* \in \{i\ | \  a_i^j>0, l_i^j=J_0\}\cup \{i\ | \  a_i^j<0, l_i^j = J_1\}$ and flip $l_{i^*}^j$ from $J_0$ to $J_1$ or from $J_1$ to $J_0$; \textbf{break} if not found
	\STATE \textbf{else} ($\exists i\in I, a_i^j <\Delta_j$, \textit{i.e.}, not a minimal cover):\\
	find $i^* \in \textup{argmin}_{I} {a_i^j}$ and set $l_{i^*}^j$ as $J_1$}
	\STATE \textbf{if} partition $\Lambda$ with adjusted labels $l$ is a minimal cover yielding partition, \textbf{break}
	\ENDFOR
	\ENDIF
	\STATE \textbf{if} lifted bilinear cover cut is generated and violated by $(x^t, y^t)$: \\
	add the lifted bilinear cover cut to the problem
	\ENDFOR
	\end{algorithmic}
\end{algorithm}

The simple randomized heuristic that we implement can be intuitively described as ``guess a reasonable partition $(I,J_0, J_1)$" and ``randomly adjust it if it does not yield a violated cut". 
{Given a solution to separate, it considers each constraint $j$ of the bilinear problem in sequence.}
For separating from a given row $j$, the heuristic must decide for each variable, whether it is in the set $I$, $J_0$, or $J_1$. 
The resulting cut is entirely determined based on this choice. 
Formally, the following steps are performed:
\begin{itemize}
\item  
The heuristic starts with an intuitive guess (steps 4-10), where the label $l_i^j$ for variable $i$ is assigned to $J_0$ if $x_iy_i$ is close to $0$ and to $J_1$ if $x_iy_i$ close to $1$. For all other indices $i$, $l_i^j$ is set to $I$ if $a^j_i >0$, and randomly to $J_0$ or $J_1$ otherwise. 

\item 
After generating the initial guess, if $(I,J_0,J_1)$ fails to be a minimal cover yielding partition, the heuristic performs several attempts at randomly adjusting the partition (steps 12-19), where labels $l_i^j$ are flipped based on different cases of failure (steps 14, 15, 16 respectively). We allow the number $S_j$  of attempts taken to depend on the row $j$, as we expect separation to be harder to perform in denser rows.

\item 
{If a partition $(I,J_0,J_1)$ is found that corresponds to a violated lifted cover, the associated inequality is added, and we proceed to attempting to generate a cut from the following problem row. }
\end{itemize}

Algorithm~\ref{alg:heu} is a basic guess-and-adjust heuristic, where the adjust step shares similarity to the WALK SAT approach of \cite{schoning1999probabilistic}. 
Nevertheless, our experiments presented in the later sections will illustrate the power of this basic heuristic, thereby also demonstrating the potential of the lifted bilinear cover cuts. 

\section{Experimental setup.}\label{sec:setup}
In this section, we present the way we setup our experiments, including instance generation, environment, and various other parameters.

\subsection{Randomly generated instances}
We consider instances of the form
\begin{alignat*}{5}
    \textup{min}\quad & \sum_{i \in [n]} \left( c_i^xx_i + c_i^y y_i\right)\\
    \textup{s.t.}\quad & \sum_{i \in [n]} a_i^j x_iy_i \geq d_j,&\quad &  \forall j \in [m],\\
    & x_i, y_i \in [0, 1],&& \forall i \in [n].
\end{alignat*}
Inspired by the classical paper~\cite{crowder1983solving} and by the fact most instances in practice are sparse, we generate sparse instances of the above model.

We generate test instances based on three parameters: 
the number of rows (\textit{i.e.}, bilinear constraints) $m$, 
the number of variables $n$, and 
the density parameter $p$. 
Moreover, we generate two groups of instances, with the first group having non-negative coefficients $a_i^j$ (which we call \textit{\nni}), and the second group having both non-negative and non-positive coefficients $a_i^j$ (which we call \textit{\msi}).

After fixing the three parameters $m$, $n$, and $p$, we randomly generate the objective coefficients $c$, the bilinear coefficients $a$ and the right-hand-sides $d$. 
The coefficients $c_i^x$ and $c_i^y$ are generated independently from a uniform distribution on $[0,1]$; $a_i^j$ is set to be non-zero with probability $p$, and when non-zero, it is generated from a uniform distribution on $[0,1]$ for \nni, or from a uniform distribution on $[-1,1]$ for \msi. 
For $d_j$, we  let $s_j := \sum_{i \in [n]} a_i^j$ and we set $d_j = r_js_j$, where $r_j$ is drawn from a uniform distribution on $[0,1]$ when $s_j > 0$ and $r_j$ is drawn from a uniform distribution on $[1,2]$ otherwise.

For our experimental design, we chose $m, n\in[100, 250, 500]$ and $p\in[0.01, 0.02, 0.05]$.  To make the problem ``reasonably" sparse (\textit{i.e.}, neither too sparse nor too dense), we only select those combinations where $np\in[5, 20]$ where $np$ can be understood as the expected number of non-zero terms for each constraint.  There is a total of $15$ different settings of $m$, $n$, and $p$ that satisfy this requirement.  For each setting, we randomly generate ten different instances with non-negative coefficients and ten instances with mixed signs. 

\subsection{Environment.}
The experiments were implemented in Python 3.9.7 with Gurobi 9.5.0 for both convex SOCP solver \texttt{convexsolver} and non-convex nonlinear solver \texttt{globalsolver}, with parameter \verb|NonConvex=2| for the latter. 
All experiments were run in parallel on the high-performance computational cluster of ISyE, Georgia Tech, which contains roughly 2,340 cores of x86-64 processing with over 28.9TB of memory spread across the systems. 
Each task was processed on a dedicated core with 8GB of memory.

\subsection{Testing.}
We compare two settings for evaluating the efficacy of the lifted bilinear cover inequalities:
\begin{itemize}
\item 
We separate the lifted bilinear cover inequalities as described in the previous section. 
Then we give as input to Gurobi the separable bilinear optimization problem augmented with these separated cuts. 
\item 
We directly input the separable bilinear optimization problem to Gurobi.
\end{itemize}
In both the above settings, we use the default parameters of Gurobi. 

For \nni, we also add a comparison to the valid inequality described in \cite{tawarmalani2010strong}:
\begin{equation}
\label{eq:MT}
\sum_{i \in [n]} \sqrt{a_i^jx_iy_i}\geq \sqrt{d_j},
\end{equation}
for the $j$-th row, where the $a^j_i$s are non-negative. 
Together with nonnegativity constraints, inequality \eqref{eq:MT} gives the convex hull of the $j$-th row if there is no upper bound on the variables. 
For testing these cuts, we add them to the problem (effectively bolstering the McCormick relaxation at the root node) and then solve the resulting problem with default Gurobi.

We note that there is no obvious way to use the cuts \eqref{eq:MT} for \msi, as there is no easy way to complement the variables in \msi~so as to produce equivalent bilinear instances with non-negative coefficients.

In Algorithm~\ref{alg:scheme}, we set both the heuristic time limit and the time limit for Gurobi non-linear non-convex solver to 1,800 seconds (half an hour). 
While it might seem that our algorithm has an unfair advantage with extra time for heuristic, we will see later that the difference is negligible as the time spent on the separation heuristic is significantly smaller compared to the time for solving the non-convex problem. 

{Moreover, for our heuristic, we set $\epsilon_z = 5\times 10^{-3}$, $\epsilon = 10^{-2}$, $T=10\cdot np$ and $S_j=10\cdot \#(a_i^j\ |\ a_i^j\neq 0)$, where $\#(a_i^j\ |\  a_i^j\neq 0)$ represents the number of non-zero coefficients in the $j$-th constraint.}

The main measure we use for comparison is the quality of the lower bound, which is measured by relative gap closed $\rho$. 
Using $z_{Mc}$ as the value of McCormick relaxation, and $z_{opt}$ as the best primal value (best solution found among both settings), we define the relative gap closed in percentage as
$$
\rho := \frac{z^* - z_{Mc}}{z_{opt} - z_{Mc}} \times 100\%,
$$
where $z^*$ is the final lower bound achieved by the corresponding method. 

Similarly, we consider the relative gap improvement $\Delta\rho$ to measure the gain by using lifted bilinear cover cuts, which is defined as
$$
\Delta\rho := \frac{z^*_{BC} - z^*_{GRB}}{z_{opt} - z_{Mc}} \times 100\% = \rho_{BC} - \rho_{GRB},
$$
where $z^*_{BC}$ is the final lower bound achieved by Gurobi for the model where the lifted bilinear cover cuts are added, whereas $z^*_{GRB}$ is the lower bound returned by the default Gurobi non-convex solver without any lifted bilinear cover cuts.

\section{Evaluation.}\label{sec:eval}
In this section, we evaluate the performance of lifted bilinear cover cuts separated with the heuristic, compared to the default Gurobi solver. 
Our main focus is the relative gap closed $\rho$ and the relative gap improvement $\Delta\rho$ defined in the previous section. 
We will however also discuss the number and the quality of bilinear cover cuts produced as well as associated computation times.

\subsection{\Nni.}
\subsubsection{Relative gap closed at the end of the time limit.}

We begin by presenting in Figure~\ref{fig:rhoposp} results regarding the relative gap closed ($\rho$) for \nni.

\begin{figure}
    \centering
    \begin{subfigure}{\textwidth}
    \centering
    \includegraphics[width=\textwidth]{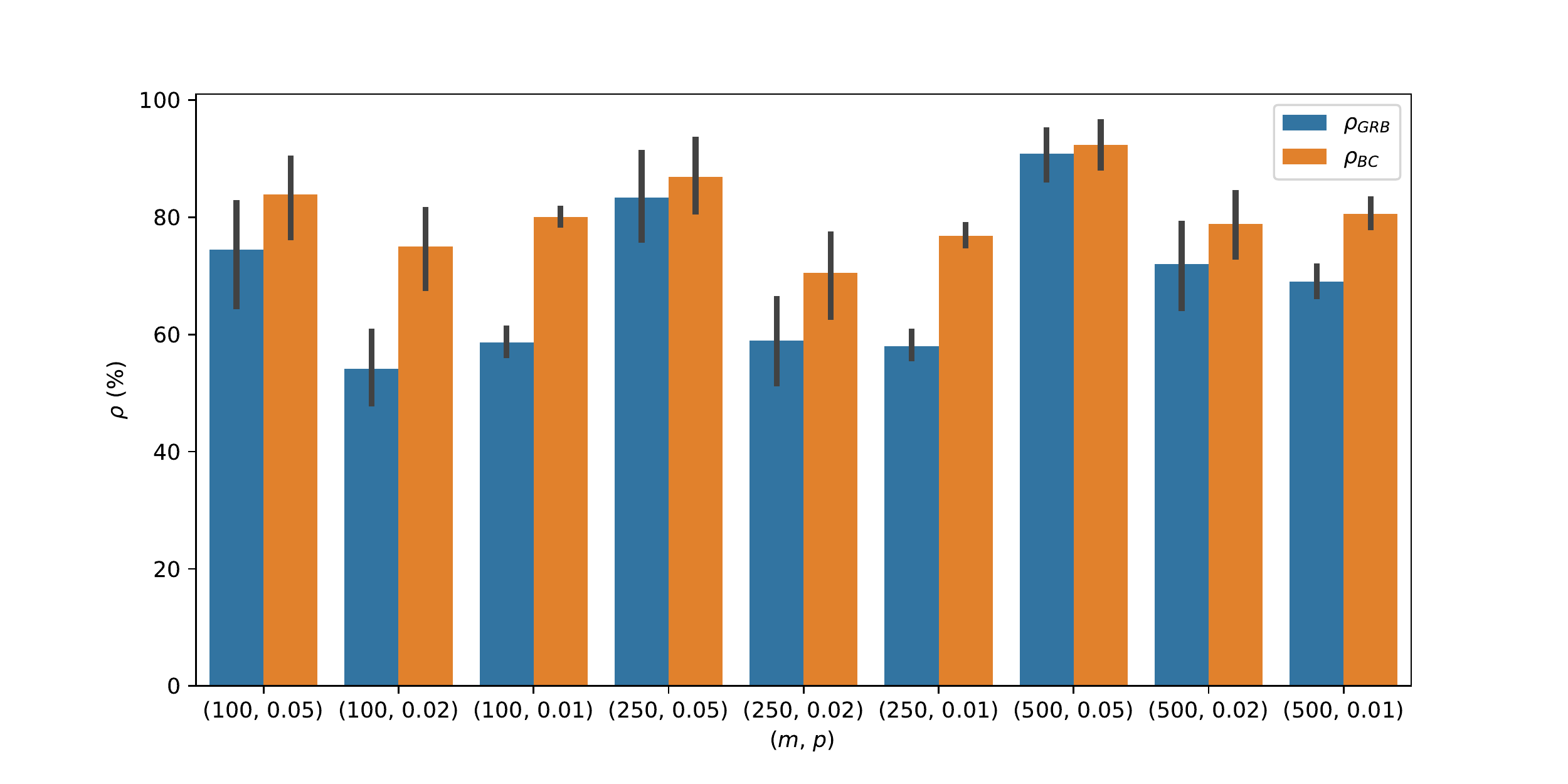}
    \end{subfigure}
    \caption{Relative gap closed $\rho$ (\%) for \nni:
    $\rho_{GRB}$ is the gap closed by Gurobi on the separable bilinear instances, and $\rho_{BC}$ is the gap closed by Gurobi on the separable bilinear instances when lifted bilinear cover cuts are added at the root node.}
    \label{fig:rhoposp}
\end{figure}

As can be seen in the figure, with the bilinear cover cuts applied at the root node, Gurobi closes more gap than the default method, resulting in a final lower bound much closer to the primal value. 
The improvement reduces when the density $p$ increases, which is anticipated as it becomes more difficult to generate violated lifted bilinear cover cuts using our separation heuristic.

In Figure~\ref{fig:rhopospm}, we take a closer look at the relative gap improvement $\Delta\rho$.
Not only is the average $\Delta\rho$ (as shown in the figure) positive, it is in fact positive for most instances. 

\begin{figure}
    \centering
    \includegraphics[width=\textwidth]{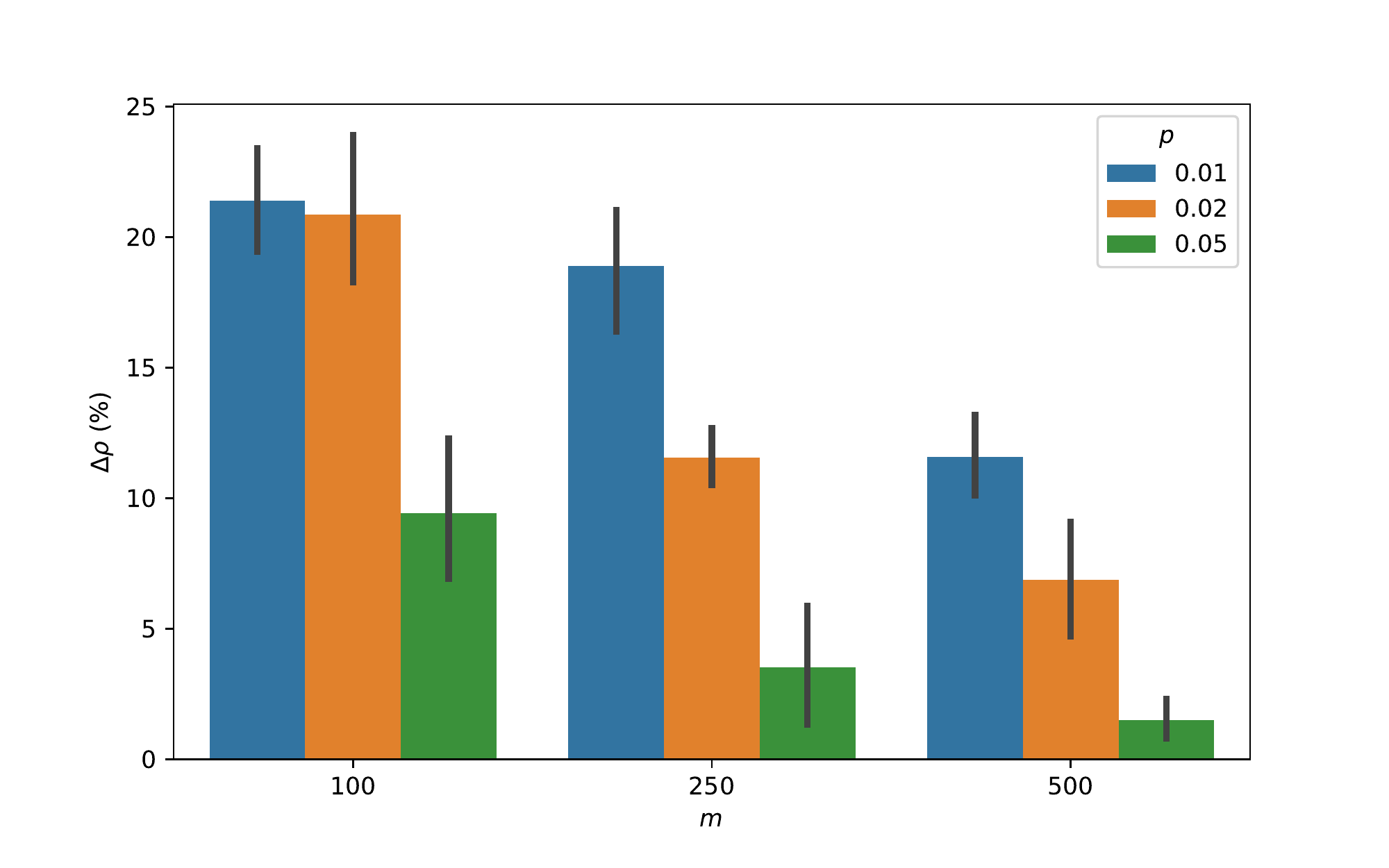}
    \caption{Relative gap improvement $\Delta\rho$ (\%) for \nni.}
    \label{fig:rhopospm}
\end{figure}

From analyzing these two figures, we draw the conclusion that lifted bilinear cover cuts provide a stable and robust performance boost for sparse \nni.

\subsubsection{Lifted bilinear cover cuts - number and root gap closed.}

Next, we investigate the number and the quality of generated lifted bilinear cover cuts. 
In Figure~\ref{fig:npospm}, we present the number of cuts generated. 
In Figure~\ref{fig:rhoheupospm}, we display the relative gap closed solely at the root node by lifted bilinear cover cuts separated using the heuristics, which we compute as 
$$\rho_{Heu}: = \frac{z_{root BC} - z_{Mc}}{z_{opt} - z_{Mc}}.$$ 

We observe that the number of separated lifted bilinear cover cuts reduces when $p$ increases.
Further, as can be anticipated, the number of cuts generated increases with the number of constraints.

The results in Figure~\ref{fig:rhoheupospm} illustrate the strength of lifted bilinear cover cuts over \nni, where much of the gap can be closed by solving the convex relaxation obtained after applying the separation heuristic, which is significantly cheaper than solving the original problem. 
Interestingly, as far as the root gap closed, the performance of these cuts appears to be independent of the sparsity level $p$: it is approximately $60\%$ for most classes of instances.

\begin{figure}
    \centering
    \includegraphics[width=\textwidth]{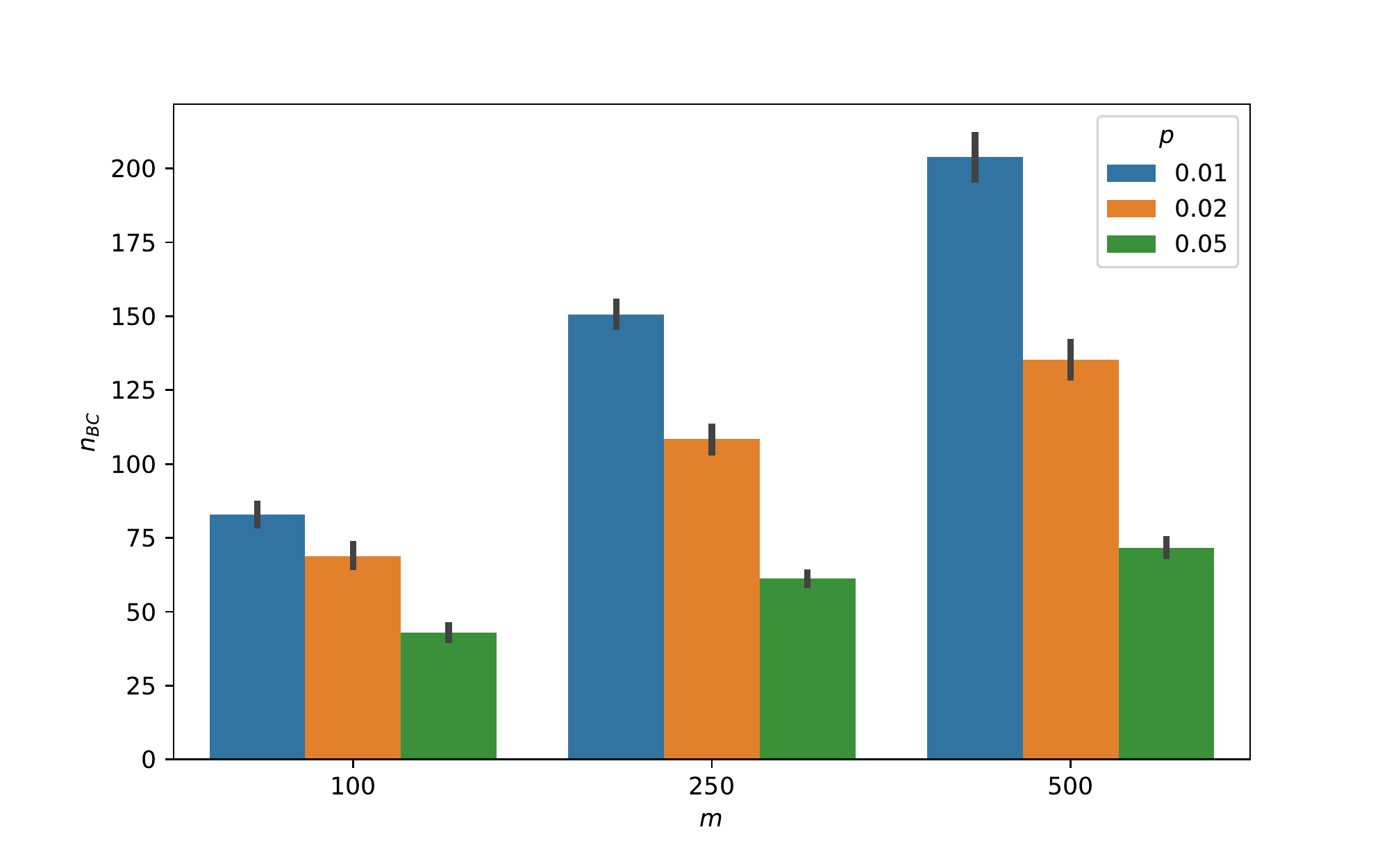}
    \caption{Number of lifted bilinear cover cuts $n_{BC}$ for \nni.}
    \label{fig:npospm}
\end{figure}

\begin{figure}
    \centering
    \includegraphics[width=\textwidth]{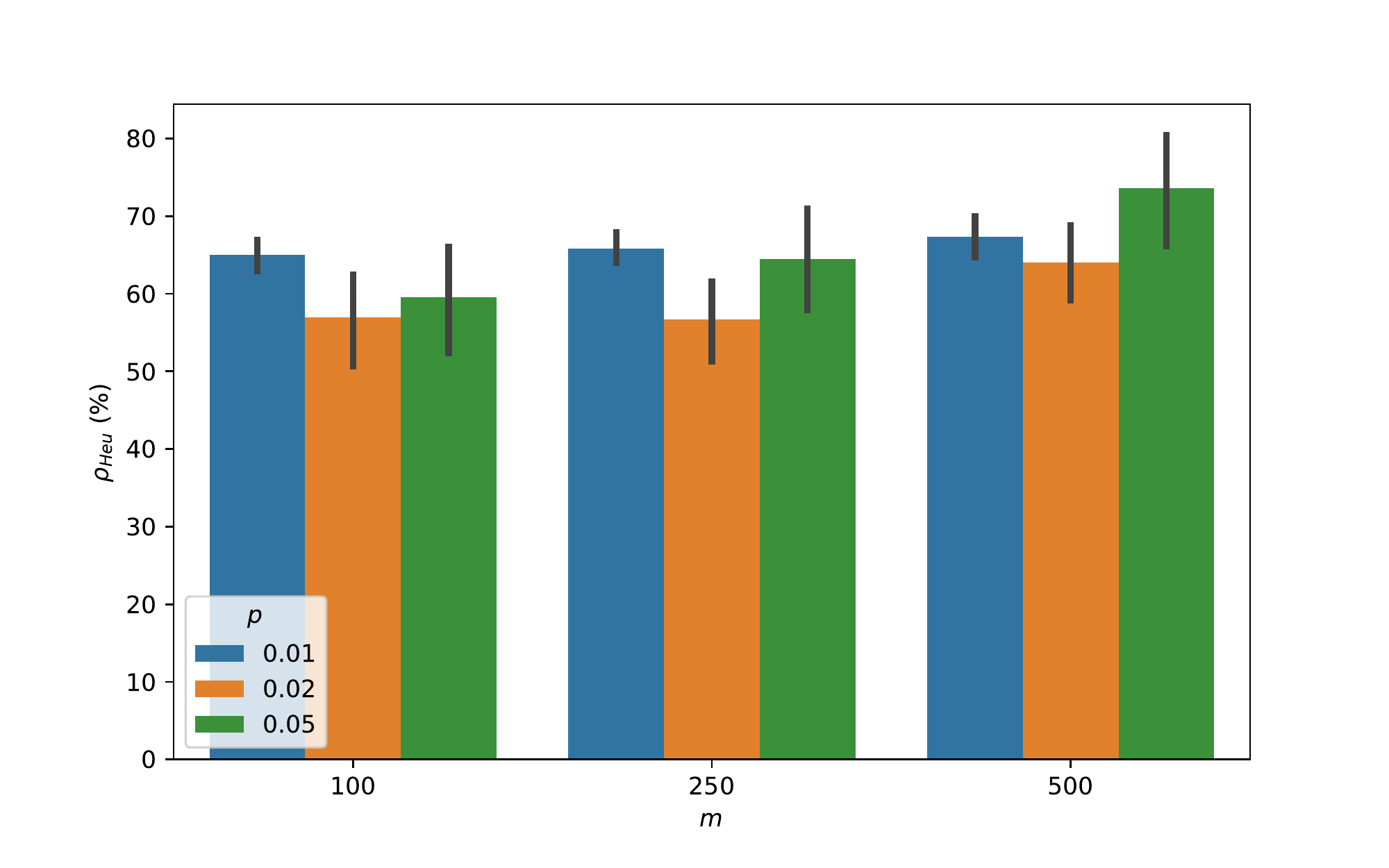}
    \caption{Relative gap closed at root node by lifted bilinear cover cuts ($\rho_{Heu}$) for \nni.}
    \label{fig:rhoheupospm}
\end{figure}

\subsubsection{Time.}

We consider now the time taken by the heuristic to separate cuts at the root node. 
As mentioned earlier, we impose a time restriction of 30 minutes for this step. 
As it turns out, the heuristic (including cut separation and solution of the SOCR convex relaxations) takes only a few seconds; 
see results in Figure~\ref{fig:tpospm}.

The time taken by the heuristic increases as the size of the instances becomes larger, though it is still relatively negligible compared to the time spent by the non-convex solver.

\begin{figure}
    \centering
    \includegraphics[width=\textwidth]{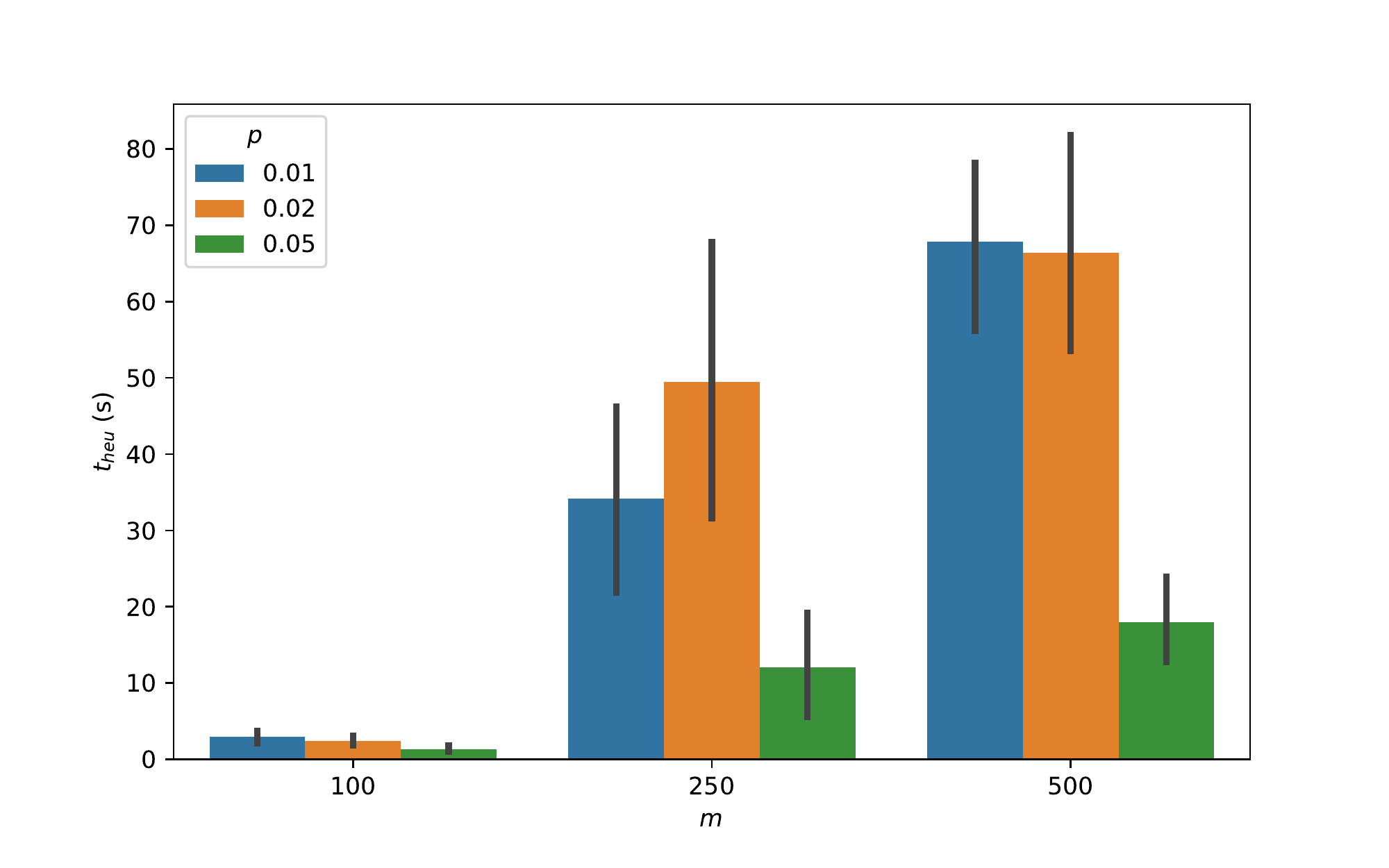}
    \caption{Heuristic CPU time (s) for \nni.}
    \label{fig:tpospm}
\end{figure}

Next we present the total times in Figure~\ref{fig:tposp}.  The difference in time is nearly non-existent, as the root node heuristic takes very little time and the instances remained unsolved after 30 minutes of Gurobi's run time in both settings.  The similarity in overall time supports the fairness of the comparison of the two settings and confirms the effectiveness of the lifted bilinear cover cuts.

\begin{figure}
    \centering
    \includegraphics[width=\textwidth]{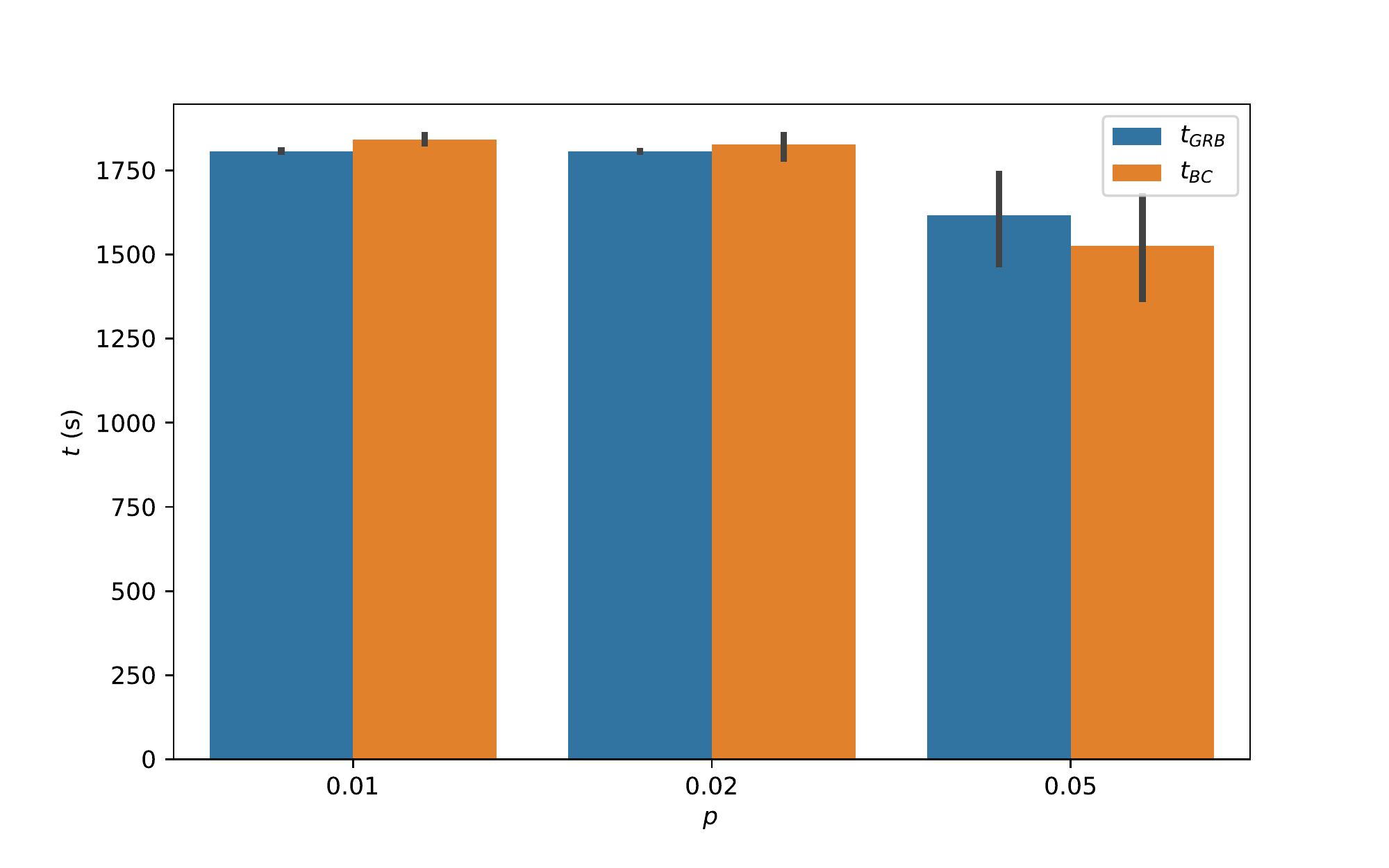}
    \caption{Total CPU time (s) for \nni: $t_{BC}$ is the time for the case where lifted bilinear cover cuts are added at the root node, and $t_{GRB}$ is the time for the other setting.}
    \label{fig:tposp}
\end{figure}

\subsubsection{Comparison with the cuts~\eqref{eq:MT}.}

We next compare the lifted bilinear cover cuts with the cuts~\eqref{eq:MT} described in \cite{tawarmalani2010strong}. 
Because a single one of these cuts can be generated from each row, no involved separation is required.
The results are presented in Figure~\ref{fig:rhopospmSP1} and Figure~\ref{fig:rhopospmSP2}.

In Figure~\ref{fig:rhopospmSP1}, we present the relative gap closed results at the root node. While it has already been shown from previous sections that the lifted bilinear cover cuts could close a significant portion of gap at the root node, we note that \eqref{eq:MT} closes a significantly smaller portion of the gap at root node. 

In Figure~\ref{fig:rhopospmSP2}, the final relative gap closed results are presented.
we observe a similar performance between default Gurobi ($\rho_{GRB}$) and the method applying \eqref{eq:MT} ($\rho_{MT}$), with lifted bilinear cover cuts closing more gap ($\rho_{BC}$).

Combining the above observations, it is clear that the lifted bilinear cover cuts are able to close more gap at the root node and provide a more robust and stable performance boost overall, compared against both default Gurobi as well as cuts~\eqref{eq:MT}.

\begin{figure}
    \centering
    \includegraphics[width=\textwidth]{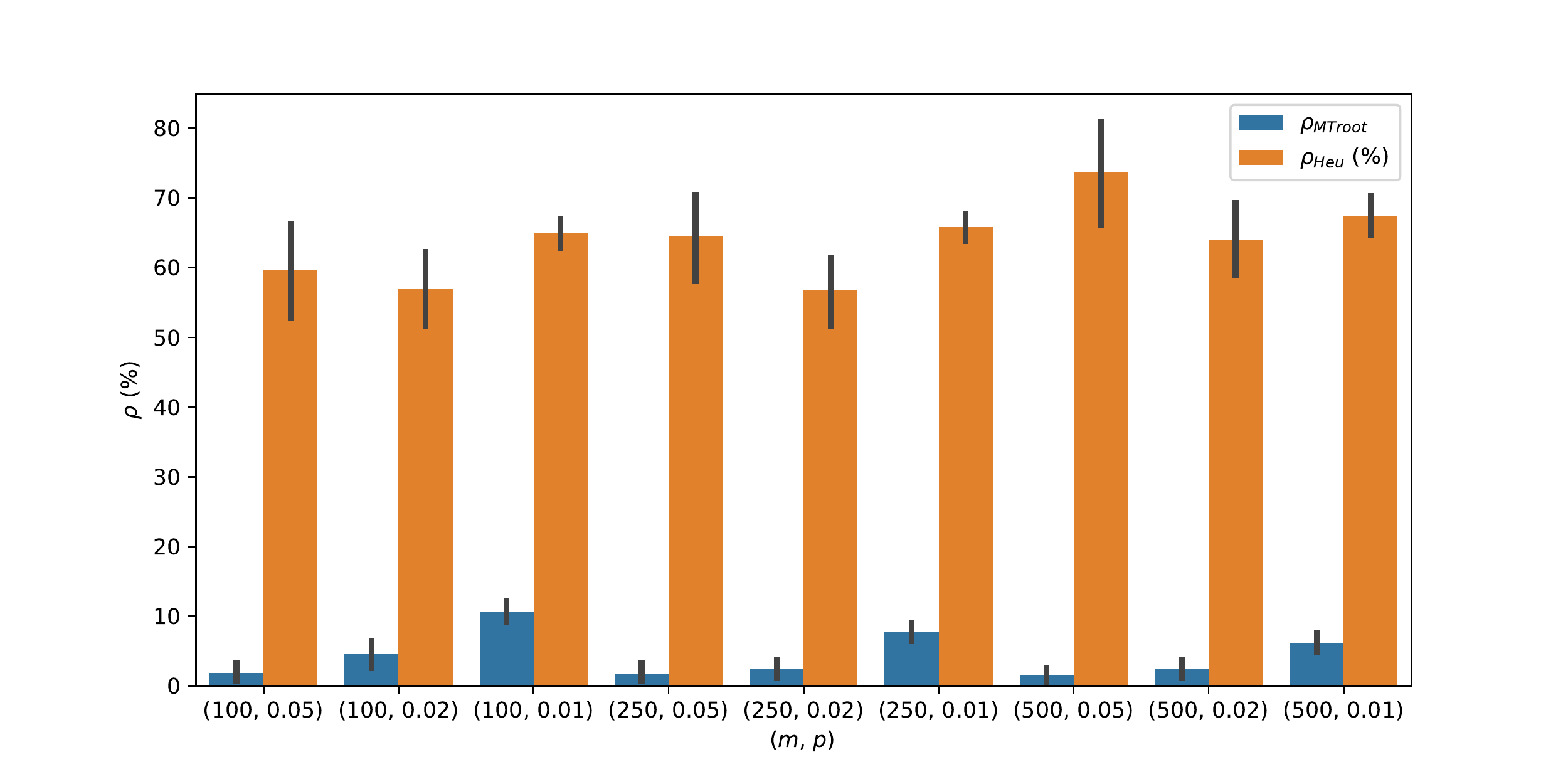}
    \caption{Relative gap closed at root node by lifted bilinear cover cuts ($\rho_{Heu}$) or by \eqref{eq:MT} ($\rho_{MTroot}$) for \nni.}
    \label{fig:rhopospmSP1}
\end{figure}

\begin{figure}
    \centering
    \includegraphics[width=\textwidth]{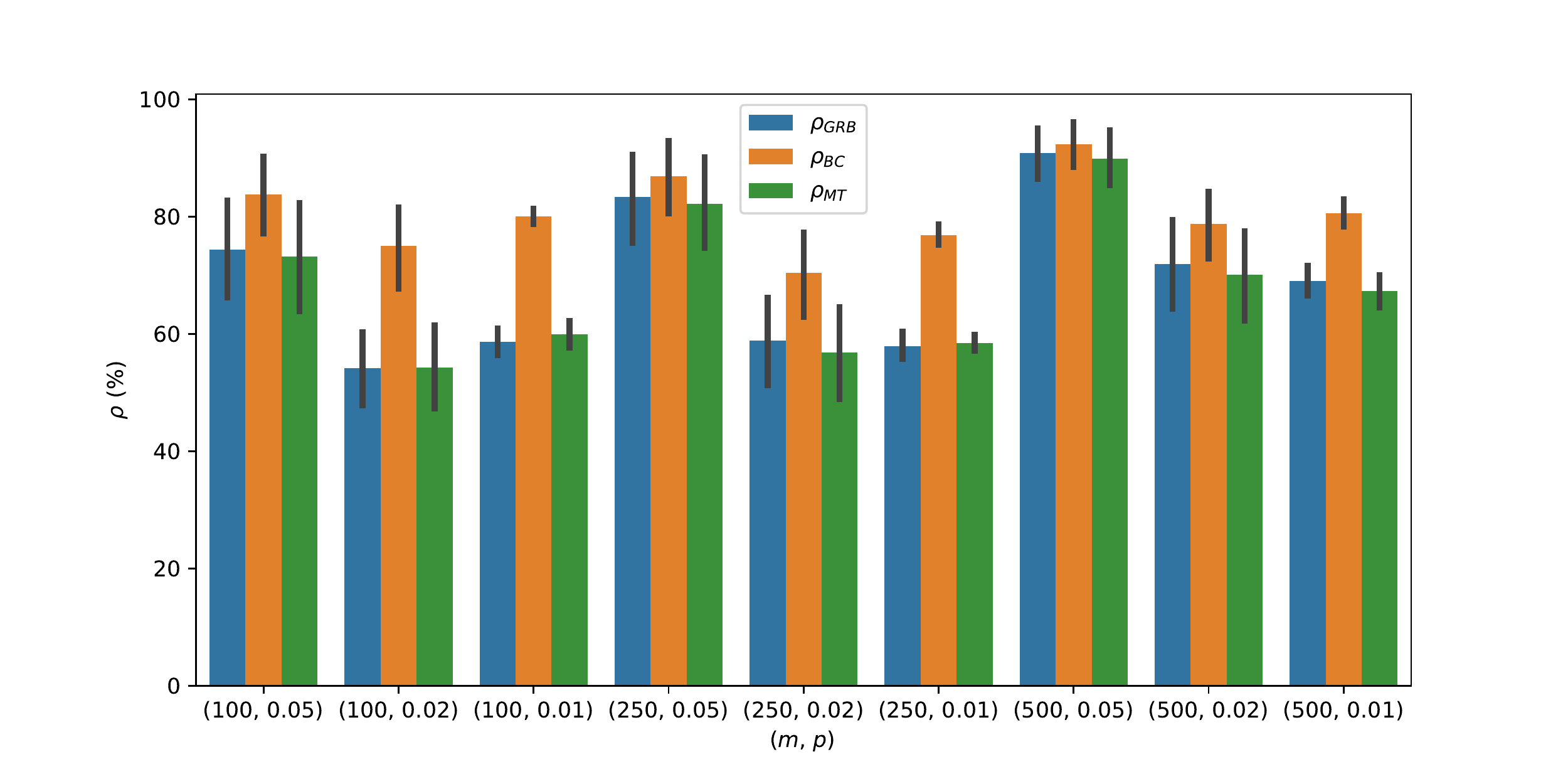}
    \caption{Relative gap closed $\rho$ (\%): $\rho_{GRB}$ is the gap closed by Gurobi, $\rho_{BC}$ is the gap closed by Gurobi when lifted bilinear cover cuts are added at the root node, and $\rho_{MT}$ is the gap closed by Gurobi when cuts \eqref{eq:MT} are added at the root node}
    \label{fig:rhopospmSP2}
\end{figure}

\subsection{\Msi.}
We now focus on \msi.  We expect that such instances are inherently harder compared to \nni, since it is more difficult to generate lifted bilinear cover cuts for them.  Nevertheless, we next show that our heuristic still works well in separating lifted bilinear cover cuts in this case. 

\subsubsection{Relative gap at the end of the time limit.}

We first present in Figure~\ref{fig:rhonegp} the results regarding the relative gap closed $\rho$.

\begin{figure}
    \centering
    \begin{subfigure}{\textwidth}
    \centering
    \includegraphics[width=\textwidth]{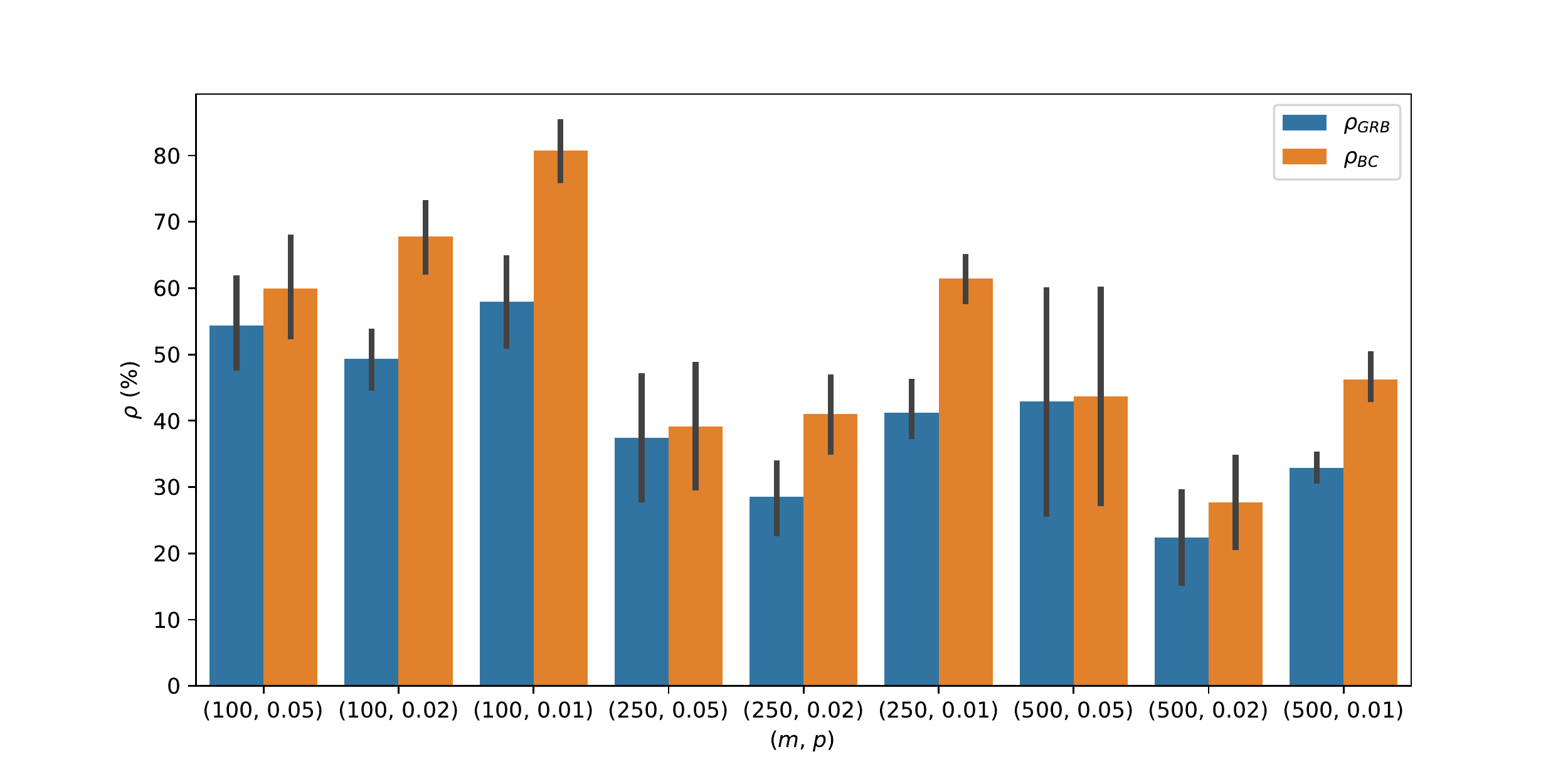}
    \end{subfigure}
    \caption{Relative gap closed $\rho$ (\%) for \msi: $\rho_{GRB}$ is the gap closed by Gurobi on the separable bilinear instances, and $\rho_{BC}$ is the gap closed by Gurobi on the separable bilinear instances when lifted bilinear cover cuts are added at the root node.}
    \label{fig:rhonegp}
\end{figure}

For the more difficult \msi, our algorithm with root node heuristic again showed improvement, which was especially marked for the sparser instances. 

\begin{figure}
    \centering
    \includegraphics[width=\textwidth]{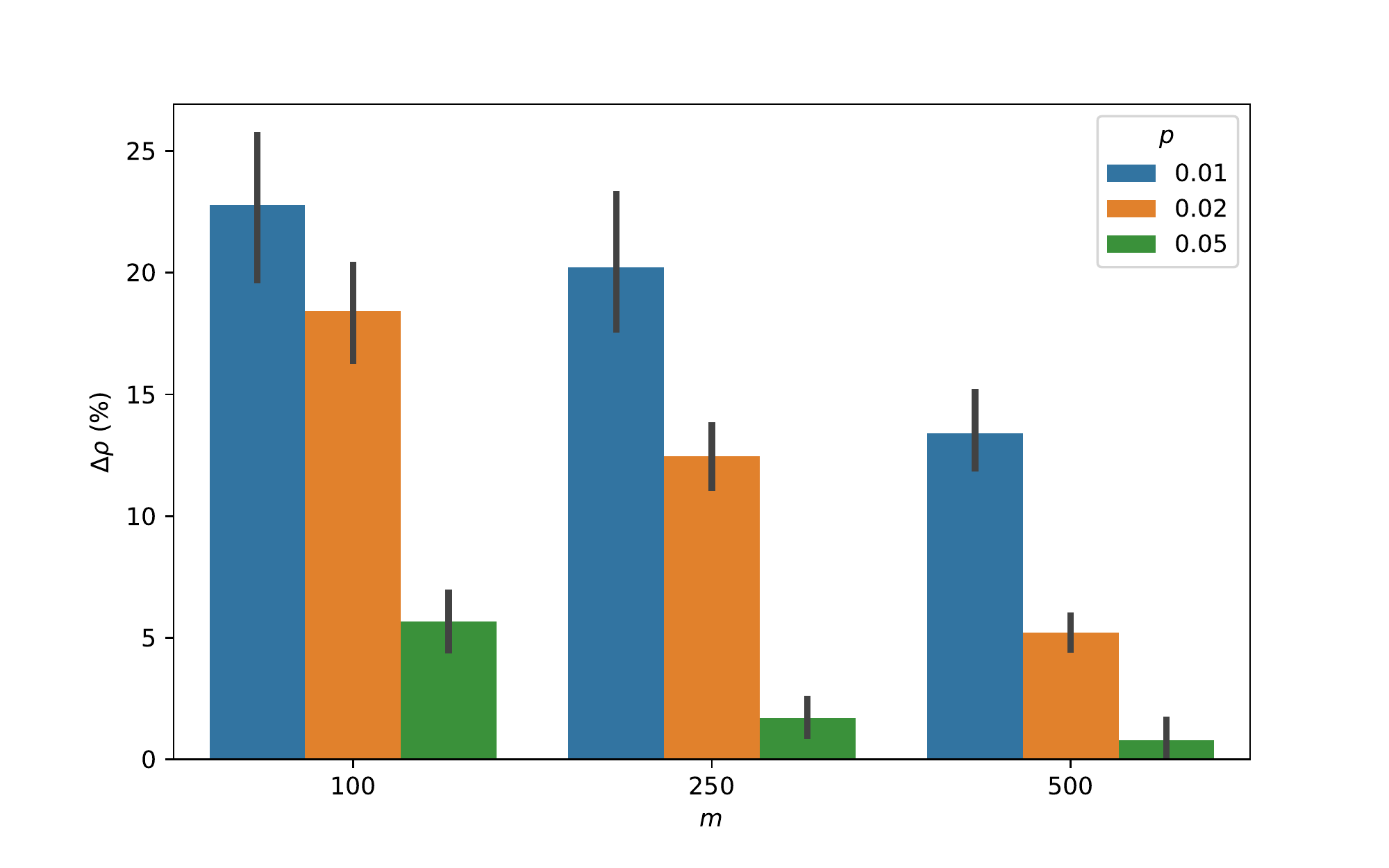}
    \caption{Relative gap improvement $\Delta\rho$ (\%) for \msi.}
    \label{fig:rhonegpm}
\end{figure}

A closer look at Figure~\ref{fig:rhonegpm} confirms the benefits, as we can observe a positive $\Delta\rho$ across every instance, which is especially noticeable for sparse settings (\textit{i.e.}, small values of $p$).

Therefore, for the more difficult \msi, Figure~\ref{fig:rhonegp} and Figure~\ref{fig:rhonegpm} indicate that our root node separation heuristic retains its strength by providing robust improvements across instances.

\subsubsection{Lifted bilinear cover cuts - number and root gap closed.}

We now turn to the number of lifted bilinear cover cuts added; see Figure~\ref{fig:nnegpm}.
Figure~\ref{fig:rhoheunegpm} presents the gap closed by these inequalities at the root node.

\begin{figure}
    \centering
    \includegraphics[width=\textwidth]{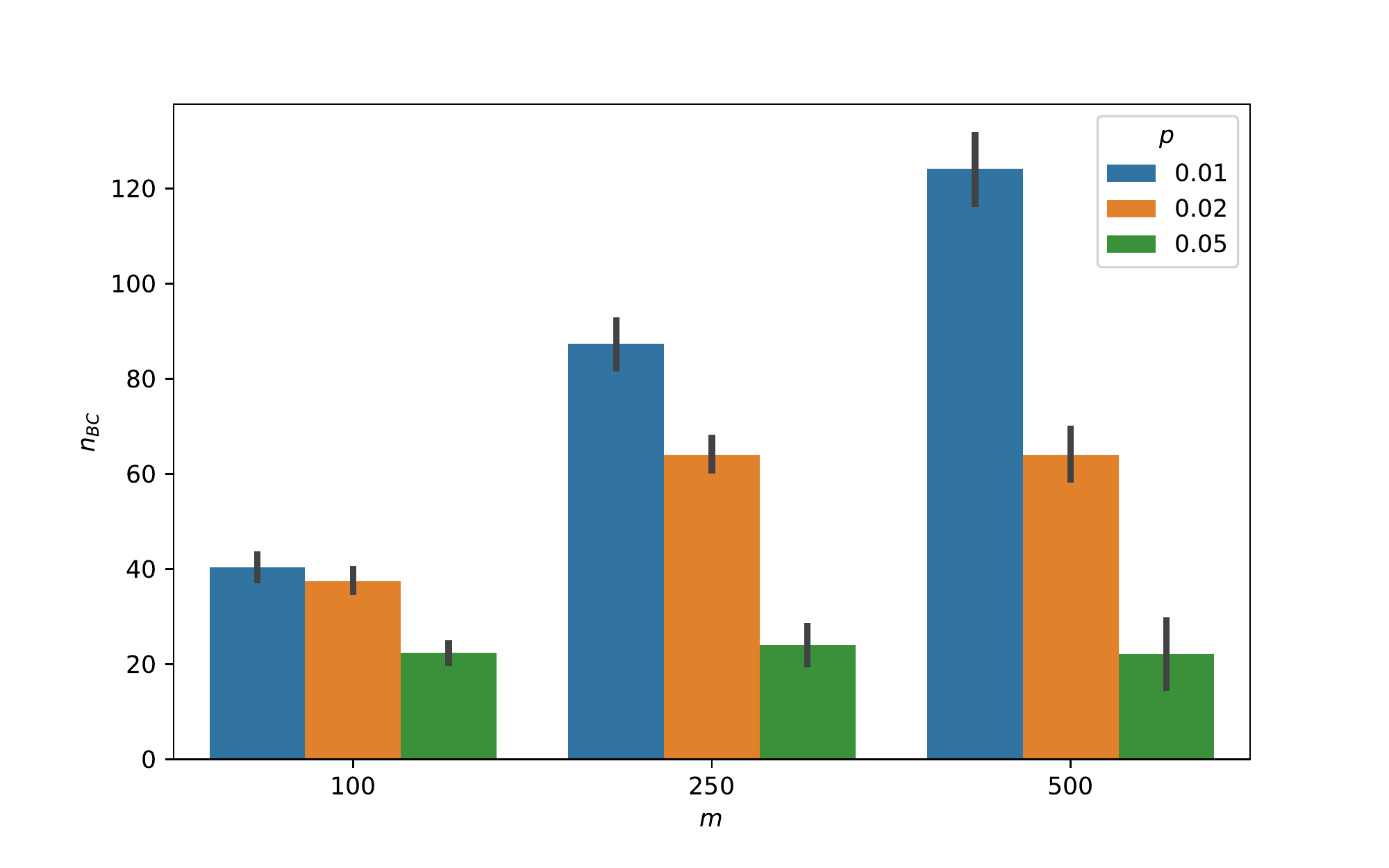}
    \caption{Number of lifted bilinear cover cuts $n_{BC}$ for \msi.}
    \label{fig:nnegpm}
\end{figure}

\begin{figure}
    \centering
    \includegraphics[width=\textwidth]{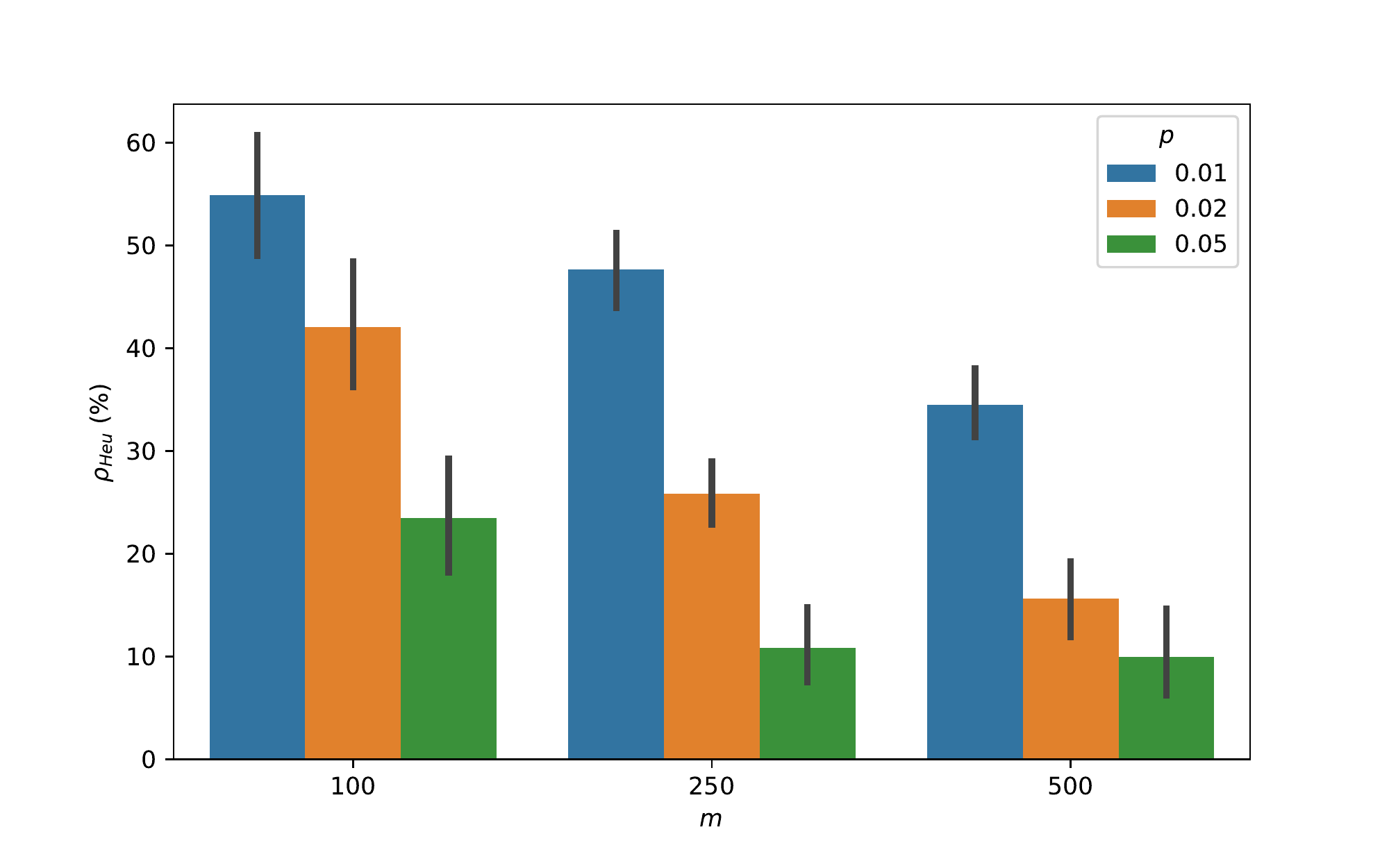}
    \caption{Relative gap closed at root node by lifted bilinear cover cuts ($\rho_{Heu}$) for \msi.}
    \label{fig:rhoheunegpm}
\end{figure}

As shown in the figures, our algorithm excels at separating lifted bilinear cover cuts for sparser problems, as more cuts were generated for smaller $p$ while closing more gap (\textit{i.e.}, larger $\rho_{Heu}$). 
Comparing Figure~\ref{fig:npospm} and Figure~\ref{fig:nnegpm}, it appears (as we expected) that negative coefficients in bilinear rows make separation harder, as evidenced by fewer cuts being generated for \msi~and the root gap closed also reducing for \msi.

\subsubsection{Time.}

We now focus on total computational time expanded. 
Figure~\ref{fig:tnegpm} presents the time for the root node separation heuristic whereas Figure~\ref{fig:tnegp} presents the overall time results for \msi.

\begin{figure}
    \centering
    \includegraphics[width=\textwidth]{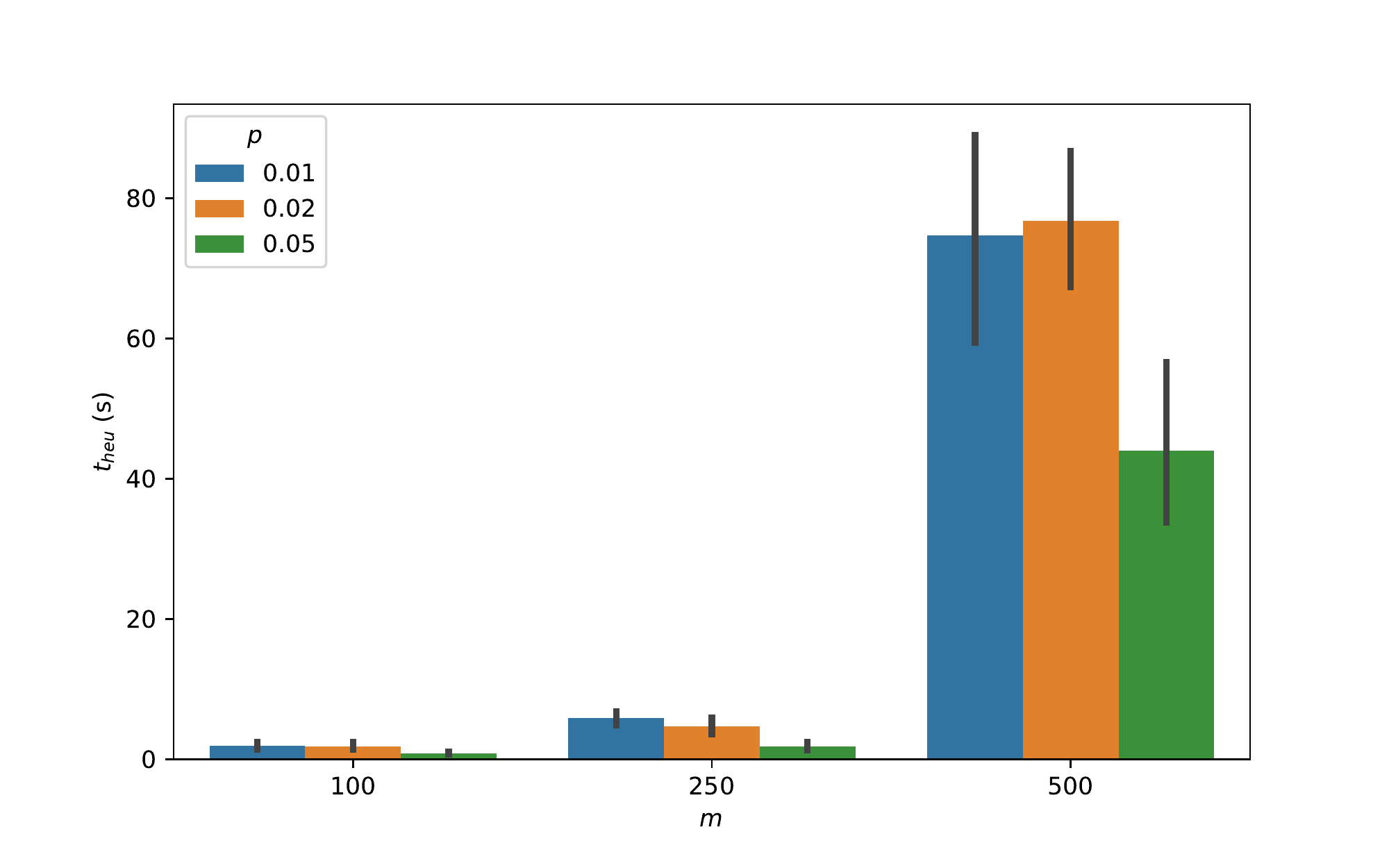}
    \caption{Heuristic CPU time (s) for \msi.}
    \label{fig:tnegpm}
\end{figure}

\begin{figure}
    \centering
    \includegraphics[width=\textwidth]{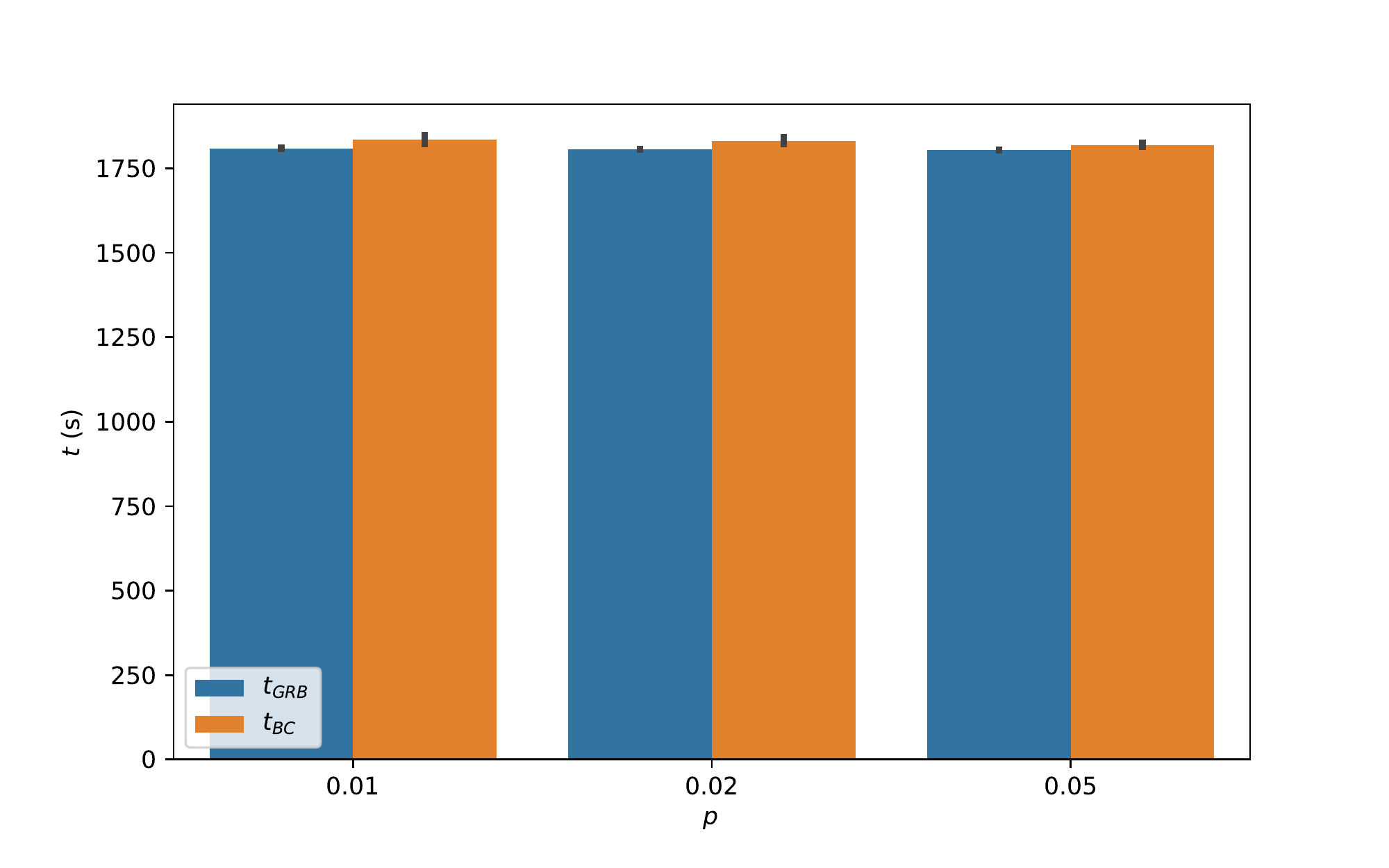}
    \caption{Total CPU time(s) for \msi.}
    \label{fig:tnegp}
\end{figure}

These numbers are very similar to the case of \nni: the heuristic itself takes very little time and most of the time is consumed by the non-convex solver.

\subsection{Overall Evaluation.}

The results for both \nni~and \msi~indicate that with minimal consumption of cpu time, the heuristic we propose is capable of separating very good lifted bilinear cover cuts, thus leading to a stable and robust performance improvement  across instances, as measured in the relative gap closed $\rho$ as well as relative gap improvement $\Delta\rho$. 
Such improvement in performance is especially noticeable when the instances are sparse.

Overall, we believe the strength and potential of the lifted bilinear cover cuts were demonstrated by our experiments, which were most conclusive for sparser problems.

\section{Conclusions and Future Directions.}\label{sec:conc}

In this paper, we designed a simple heuristic to separate the lifted bilinear cover cuts introduced in~\cite{gu2022lifting}.  In the computational experiments, the cuts demonstrate major and robust improvement for sparse problems, illustrating the potential of these cuts in situations were SDP-based formulations might not be as helpful.

We envision multiple future directions that will better uncover and utilize the potential of the lifted bilinear cover inequalities, as well as achieving better results:
\begin{itemize}
    \item 
    {\bf Complexity of separation:}
    We developed a heuristic for separation since we could not design a polynomial-time exact separation routine. 
    We conjecture that the separation problem is NP-hard although proving this result appears nontrivial.  
    \item 
    {\bf Improving the heuristic for cutting planes:}    While the heuristic we developed is quite simple, it uses randomness. 
    It would be interesting to devise a deterministic heuristic and explore other more sophisticated variants of the heuristic. Another direction is to consider separating cuts from aggregations of constraints, an approach that has proven to be useful for MILPs (see \cite{marchand2001aggregation,bodur2018aggregation}) and for QCQPs (see \cite{yildiran2009convex,burer2017convexify,modaresi2017convex,dey2022obtaining}).
    \item 
    {\bf Integrating the bilinear cover cuts within the solver:} 
    The procedure described in the paper only applies lifted bilinear cover cuts at the root node, and its potential cannot be fully realized without integration within a global solver. 
    It is clear that the method could be applied at children nodes during the branch-and-bound process; evaluating the benefits this provides is an open question.
    \item 
    {\bf Other problem types:} 
    It is easy to see that any quadratic constraint can be relaxed (at the expense of an increase in the number of variables) into a separable bilinear constraint. 
    An interesting direction of research is to determine if adding these lifted bilinear cover cuts to such relaxation is useful in the solution of nonconvex QCQPs.
\end{itemize}

\section{Acknowledgments}
Santanu S. Dey gratefully acknowledges the support by ONR under grant N000141912323.

\bibliographystyle{plain}
\bibliography{mybibliography}
\end{document}